\documentclass[12pt]{amsart}

\usepackage[margin=1in]{geometry}
\usepackage{amsthm, amsmath}
\usepackage{tikz, tikz-cd}
\usepackage{hyperref}
\usepackage[noabbrev,capitalize]{cleveref}

\usetikzlibrary{calc}

\input{xy}
\xyoption{all}

\newtheorem{theorem}{Theorem}[section]
\newtheorem{lemma}[theorem]{Lemma}
\newtheorem{corollary}[theorem]{Corollary}

\newtheorem{proposition}[theorem]{Proposition}

\theoremstyle{definition}

\newtheorem{definition}[theorem]{Definition}

\newtheorem{remark}[theorem]{Remark}

\definecolor{urlcolor}{rgb}{0,.145,.698}
\definecolor{linkcolor}{rgb}{.01,0.01,0.51}
\definecolor{citecolor}{rgb}{.12,.54,.11}
    
\hypersetup{
      breaklinks=true,
      colorlinks=true,
      urlcolor=urlcolor,
      linkcolor=linkcolor,
      citecolor=citecolor,
      }

\begin{document}

\title{Tilting modules arising from knot invariants}
\date{}
\author{Ralf Schiffler}\address{Department of Mathematics, University of Connecticut, 
Storrs, CT 06269-3009, USA}
\email{schiffler@math.uconn.edu}

\author{David Whiting}\address{Department of Mathematics, University of Connecticut, 
Storrs, CT 06269-3009, USA}
\email{david.whiting@uconn.edu}

\thanks{The first author was supported  by the NSF grant DMS-1800860 and by the University of Connecticut.}

\maketitle
\begin{abstract}
We construct tilting modules over Jacobian algebras arising from knots.
To a two-bridge knot $L[a_1,\ldots,a_n]$, we associate a quiver $Q$ with potential and its Jacobian algebra $A$. We construct a family of canonical indecomposable  $A$-modules $M(i)$, each supported on a different specific subquiver $Q(i)$ of $Q$. Each of the $M(i)$ is expected to parametrize the Jones polynomial of the knot. We study the direct sum $M=\oplus_iM(i)$ of these indecomposables inside the module category of $A$ as well as in the cluster category.

In this paper we consider the special case where the two-bridge knot is given by two parameters $a_1,a_2$. We show that the module $M$ is rigid and $\tau$-rigid, and we construct a completion of $M$ to a tilting (and $\tau$-tilting) $A$-module  $T$. We show that the endomorphism algebra $\operatorname{End}_AT$ of $T$ is isomorphic to $A$, and that the mapping $T\mapsto A[1]$ induces a cluster automorphism of the cluster algebra $\mathcal{A}(Q)$. This automorphism is of order two.
Moreover, we give a mutation sequence that realizes the cluster automorphism.
In particular, we show that the quiver $Q$ is mutation equivalent to an acyclic quiver of type $T_{p,q,r}$ (a tree with three branches). This quiver  is of finite type
if $(a_1,a_2)=(a_1,2), (1,a_2),$ or $(2,3)$, it is tame for $(a_1,a_2)=(2,4)$ or $(3,3)$, and wild otherwise.
\end{abstract}
\section{Introduction}
We construct tilting modules over Jacobian algebras that are motivated by knot invariants of two-bridge knots and their relation to cluster algebras.

A relation between knot invariants and cluster algebras has been established recently in \cite{LS}, where the authors give a realization of the Jones polynomial of a two-bridge link in terms of the Laurent expansion of a certain cluster variable. In this approach, one first constructs  a quiver $R$ of Dynkin type $\mathbb{A}$ from the link and then considers the cluster algebra $\mathcal{A}(R)$ determined by it. The cluster variable $x_M$ that realizes the Jones polynomial is the one associated with the unique indecomposable module $M$ of the path algebra of $R$ that is one-dimensional at each vertex. In this special situation, the cluster variable $x_M$ is a Laurent polynomial whose terms are parametrized by the submodules of the module $M$ \cite{CC}. Equivalently, it can be computed via a perfect matching formula of an associated snake graph \cite{MSW}. In particular, the number of terms (counted with multiplicities) in the Jones polynomial is  equal to the number of submodules of $M$ and equal to the number of perfect matchings of the snake graph.

We point out that a similar result for the Alexander polynomial was given more recently in \cite{NT} using ancestral triangles, and for the HOMFLY polynomial in \cite{Y} using path posets. See also \cite{MO} for an interesting reformulation of the results in \cite{LS}.

On the other hand, the authors of \cite{CDR} established a dimer model for links, using the (overlaid) Tait graph $G$ of the link, to interpret the Alexander polynomial as a partition function. In this method, one has to remove  two adjacent regions from the graph $G$ and thereafter one can compute the Alexander polynomial as the determinant of the (weighted) adjacency matrix of the reduced graph. This formula is reminiscent of Alexander's original definition of the polynomial as the determinant of the incidence matrix from which one removes two columns that correspond to two adjacent regions in the graph \cite{A}. 

It is crucial to note that the result does not depend on the choice of the two adjacent regions in $G$, and this observation motivated the research project in this paper.

If the link is  a two-bridge link, we can associate a quiver $Q$ with potential $W$ to the graph $G$, and we denote by $A$ the corresponding Jacobian algebra. The removal of two adjacent regions in $G$ corresponds to the removal of a connected subquiver of $Q$, and, for a certain choice of   adjacent regions, the resulting quiver is precisely the quiver $R$ used in \cite{LS} to compute the Jones polynomial of the link. 
If we chose a different pair of adjacent regions then the subquiver may look very different, but it is natural to expect that a variation of  the methods of \cite{LS} will also apply for these subquivers and give another way to compute the Jones polynomial. This problem will be considered in a different paper.

In this paper, we want to consider all possible  choices of adjacent pairs at the same time. Each such choice $i$ will determine a subquiver $Q^i$ of the quiver $Q$ which we need to remove. On the remaining quiver $Q(i)=Q\setminus Q^i$, we define a canonical indecomposable $A$-module $M(i) $ of dimension 1 at every vertex of $Q(i)$, and we let $M=\oplus_i M(i)$ be the direct sum of these indecomposables. 

Motivated by the results mentioned above, we want to study the structure of $M$ inside the module category of $A$. Thereby we combine all possible choices of adjacent regions into one object. It is natural to expect that this module $M$ has a particularly nice structure.

By construction, $M$ is a sincere $A$-module, meaning it is supported at every vertex of $Q$,  each indecomposable summand $M(i)$ of $M$ is supported on a different subquiver $Q(i)$ of $Q$ and each $M(i)$ should parametrize the Jones and the Alexander polynomial of the two-bridge link via its submodules.

Recall that a two-bridge link $L[a_1,a_2,\ldots,a_n]$ is given by a sequence of positive integers $a_1,a_2,\ldots,a_n$ and the number of terms in the Jones (and Alexander) polynomial is equal to the numerator of the continued fraction 
\[[a_1,a_2,\ldots,a_n]= a_1+\cfrac{1}{a_2+\cfrac{1}{\ddots +\cfrac{1}{a_n}}}.\]

In this paper, we consider the special case where $n=2$.
Thus we have two positive integers $a_1,a_2$ and the number of terms in the Jones polynomial is $a_1a_2+1$. The case $n=1$ is contained in this case because the link $L[a_1]$ is equivalent to the link $L[1,a_1-1]$. 
For $n=2$, the quiver $Q$ is of the form illustrated in Figure~\ref{fig:Q}. In this situation the removal of two adjacent regions in $G$ corresponds to the removal of a vertex $r_i$ on the oriented cycle of $Q$ together with the branch at the vertex $r_i$ if $i =0$ or $i=a_2$. We let $M(r_i)$ be the unique indecomposable $A$-module that is one-dimensional at every vertex of the remaining subquiver. We prove that each $M(r_i)$ has precisely $a_1a_2+1$ submodules as expected.

We then construct an indecomposable $A$-module $M(x)$ for each of the remaining vertices $x=s_1,s_2,\ldots,s_{a_1-1},t_1,t_2,\ldots,t_{a_1-1}$ of $Q$, and we define $T=\oplus_{x\in Q_0} M(x)$ to be the direct sum of all these indecomposables.

We show that $T$ is a tilting $A$-module. This means that $\operatorname{Ext}^1_A(T,T)$ vanishes, that the projective dimension of $T$ is at most one and that there exists a short exact sequence of the form $0\to A\to T^0\to T^1\to 0$ with $T^0,T^1$ in the additive closure of $T$.
Tilting modules and their endomorphism algebras play a central role in representation theory, see for example \cite{AsSiSk, handbook}. 

Furthermore, we show that $T$ is a $\tau$-tilting $A$-module as defined in \cite{AIR}. This means that in addition to being a tilting module, $\operatorname{Hom}(T, \tau T) = 0$ where $\tau$ is the Auslander-Reiten translation.

Our $\tau$-tilting module $T$ is particularly nice, since we can show that its endomorphism algebra $\operatorname{End}_A T$ is isomorphic to the dual of $A$, and $A$ is self-dual.

\medskip
We then consider this situation at the level of the corresponding cluster algebra $\mathcal{A}(Q)$ defined in \cite{FZ1}. We show that  the $\tau$-tilting module $T$ induces a cluster-tilting object in the cluster category $\mathcal{C}_Q$ introduced in \cite{Amiot}. It turns out that our quiver is mutation equivalent to an acyclic quiver and thus the cluster category is of acyclic type as defined in \cite{BMRRT} (and \cite{CCS}, for Dynkin type $\mathbb{A}$).
Therefore, the results in \cite{CK2} imply that each of $A[1]$ and $T$ corresponds to a cluster $\mathbf{x}_{A[1]}$ and $\mathbf{x}_T$ in the cluster algebra (here $[1]$ denotes the shift in the cluster category).
Using a result from \cite{ASS}, we see that the correspondence $T \mapsto A[1]$ induces a cluster automorphism $\sigma$ of $\mathcal{A}(Q)$, that maps the cluster $\mathbf{x}_T$  to the cluster $\mathbf{x}_{A[1]}$ and induces an isomorphism between the quiver $Q$ of the cluster $
\mathbf{x}_{A[1]}$ and the opposite of the quiver of the cluster $\mathbf{x}_T$.
We further show that $\sigma$ has order 2.
We then construct a sequence of mutations $\mu$ that realizes the automorphism $\sigma$.

Our main results are summarized in the following theorem.
\begin{theorem} Let $A$ be the Jacobian algebra of $Q=Q[a_1,a_2]$ and let $T=\oplus_{x\in Q_0} M(x)$.
\begin{enumerate}
\item 
 $T$ is a tilting $A$-module and a $\tau$-tilting $A$-module.
 \item The endomorphism algebra $\operatorname{End}_A T$ is isomorphic to $A$.
 \item The correspondence $T \mapsto A[1]$ induces a cluster automorphism $\sigma$ of order two of the cluster algebra $\mathcal{A}(Q)$.
 \item The automorphism $\sigma$ is given by the sequence of mutations \[ \mu= \mu_{\mkern-1.5mu R}^{-1}\circ\mu_S\circ\mu_T\circ\mu_{\mkern-1.5mu R}, \]
 where 
\begin{align*}
        &\mu_{\mkern-1.5mu R} = \mu_{r_{a_2}-1} \circ \cdots \circ \mu_{r_2} \circ \mu_{r_1}, \\
        &\mu_S = \mu_{s_1}\mu_{s_2}\mu_{s_3} \cdots \mu_{r_{a_2}} \circ \cdots \circ \mu_{s_1}\mu_{s_2}\mu_{s_3} \circ \mu_{s_1}\mu_{s_2} \circ \mu_{s_1}, \\
        &\mu_T = \mu_{t_{a_1-1}}\mu_{t_{a_1-2}}\mu_{t_{a_1-3}} \cdots \mu_{r_0} \circ \cdots \circ \mu_{t_{a_1-1}}\mu_{t_{a_1-2}}\mu_{t_{a_1-3}} \circ \mu_{t_{a_1-1}}\mu_{t_{a_1-2}} \circ \mu_{t_{a_1-1}}.
\end{align*}
 \item The cluster algebra $\mathcal{A}(Q)$ is of acyclic type $\mu_{\mkern-1.5mu R} \,Q = T_{p,q,r}$, a tree with three branches. Moreover, the cases where it is of finite or tame type are the following.
 \[ \begin{array}{ll}
 \mathbb{A}_{2a_1+1} & \textup{if $a_2=2$};\\
 \mathbb{D}_{a_2+1} & \textup{if $a_1=1$};\\ 
 \mathbb{E}_{6} & \textup{if $(a_1,a_2)=(2,3)$};\\
  \widetilde{\mathbb{E}}_{7} & \textup{if $(a_1,a_2)=(3,3)$};\\
\widetilde{ \mathbb{E}}_{6} & \textup{if $(a_1,a_2)=(2,4)$};
  \end{array}
 \] 
 and it is of wild acyclic type in all other cases.
 \end{enumerate}
\end{theorem}

The case $n\ge 3$  is more complicated and will require different methods. We expect that we can still construct the tilting module $T$ and a mutation sequence. However, already in small examples, we do not obtain a cluster automorphism of order two.

The paper is organized as follows.
We define the $A$-module $T$ in 
 Section~\ref{sec:def_of_M} and compute the number of submodules in 
Section~\ref{sec:submodules}. After computing the Auslander-Reiten translate of $T$ in 
Section~\ref{sec:tau_M}, we show that $T$ is a tilting and a $\tau$-tilting $A$-module in
Section~\ref{sec:tilting_module}.
Section~\ref{sec:end_M} is devoted to the study of the endomorphism algebra of $T$ and to the proof of $\operatorname{End}_A T\cong A$. In Section~\ref{sec:mutation}, we construct the mutation sequence that transforms $A$ into $T$ and show that the corresponding cluster automorphism has order two. We illustrate the results in an example in Section~\ref{sec:example}.

\begin{figure}[h]
    \centering
    \begin{tikzpicture}[->,scale=.7] 
        \node (r0) at (60:2cm) {$r_0$}; \node (ra2) at (120:2cm) {$r_{a_2}$};
        \node (ra2-1) at (180:2cm) {$r_{a_2-1}$};
        \node (r1) at (0:2cm) {$r_1$};
        
        \node (t1) at ($(r0) + (2,0)$) {$t_1$};
        \node (t2) at ($(t1) + (1.2,0)$) {};
        \node (ta1-2) at ($(t2) + (1.2,0)$) {};
        \node (ta1-1) at ($(ta1-2) + (1.6,0)$) {$t_{a_1-1}$};
        
        \node (sa1-1) at ($(ra2) + (-2.5,0)$) {$s_{a_1-1}$};
        \node (sa1-2) at ($(sa1-1) + (-1.6,0)$) {};
        \node (s2) at ($(sa1-2) + (-1.2,0)$) {};
        \node (s1) at ($(s2) + (-1.2,0)$) {$s_{1}$};
        
        \draw (s1) -- (s2);
        \draw[thick, loosely dotted, -] (s2) -- (sa1-2);
        \draw (sa1-2) -- (sa1-1);
        \draw (sa1-1) -- (ra2);
        
        \draw (r0) -- (t1);
        \draw (t1) -- (t2);
        \draw[thick, loosely dotted, -] (t2) -- (ta1-2);
        \draw (ta1-2) -- (ta1-1);
        
        \draw (165:2cm) arc (165:135:2cm);
        \draw (45:2cm) arc (45:15:2cm);
        \draw (105:2cm) arc (105:75:2cm);
        \draw (225:2cm) arc (225:195:2cm);
        \draw (-15:2cm) arc (-15:-45:2cm);
        \draw[thick, loosely dotted, -] (-55:2cm) arc (-55:-125:2cm);
    
    \end{tikzpicture}
    \caption{The Quiver $Q = Q[a_1, a_2] = (Q_0, Q_1)$}
    \label{fig:Q}
\end{figure}
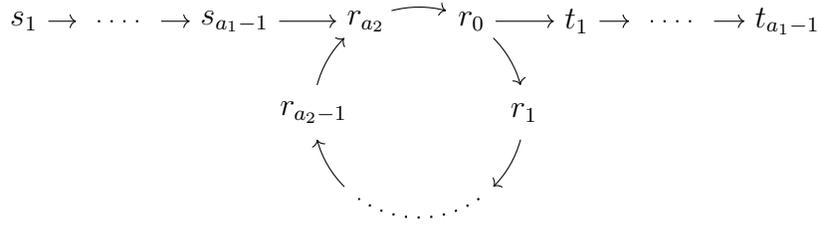

\section{Definition of \texorpdfstring{$T$}{T}.}\label{sec:def_of_M}
Throughout the paper, let $k$ denote an algebraically closed field. Let $a_1,a_2$ be positive integers and let $Q=Q[a_1,a_2]$ be the quiver shown in Figure \ref{fig:Q}. Thus $Q$ consists of an oriented cycle of length $a_2+1$ and a path of length $2a_1-1$ whose halfway arrow is part of the cycle. Let $Q_0$ be the set of vertices of $Q$ and $Q_1$ the set of arrows. We label the vertices that lie on the cycle by $r_0,r_1,\ldots, r_{a_2}$, those on the incoming branch by $s_1,s_2,\ldots, s_{a_1-1}$ 
and those on the outgoing branch by $t_1,t_2,\ldots, t_{a_1-1}$.

We equip $Q$ with the potential $W=r_0\to r_1\to \cdots \to r_{a_2}\to r_0$ and denote the corresponding Jacobian algebra by $A$. Thus $A=kQ/I$ is the quotient of the path algebra of $Q$ by the two-sided ideal $I$ generated by all subpaths of $W$ of length $a_2$.

The opposite quiver $Q^{\rm op} $ is obtained from $Q$ by reversing the direction of each arrow. Let $D=\operatorname{Hom}_k(-,k)$ denote the standard duality.
\begin{remark}
 The quiver $Q$ is isomorphic to its opposite quiver $Q^{\rm op}$ and, hence, the algebra $A$ is isomorphic to its dual $D\!A$.
\end{remark}

For every $x\in Q_0$, let $P(x),I(x)$ and $ S(x)$ denote the indecomposable projective, injective and simple $A$-module at $x$, respectively.  
Every $A$-module $L$ is described as a representation of the quiver $Q$ satisfying the relations in $I$, and we shall frequently use the notation $L=(L_x,\varphi_\alpha)_{x\in Q_0,\alpha\in Q_1}$, where $L_x$ is the $k$-vector space at vertex $x$ and $\varphi_\alpha$ is the linear map on the arrow $\alpha$ of the representation corresponding to $L$.
For further details on representation theory we refer to \cite{AsSiSk,Sbook}.
\smallskip

In this section, we define an indecomposable $A$-module $M(x)$ for every vertex $x\in Q_0$, and denote by $T$ the direct sum $T=\oplus_{x\in Q_0} M(x)$. 
The definition of $M(x)$ is given in three separate cases depending on $x$ being a vertex in the cycle or in one of the two branches. 


\begin{definition}\label{def:M(ri)}
    For all vertices labeled $r_i$, we define the module $M(r_i) = (M(r_i)_x, \varphi_\alpha)$ by
    \begin{align*}
    &M(r_i) = 
        \begin{smallmatrix}
            r_{i+1} & & s_1 \\
            r_{i+2} & & s_2 \\
            \vdots & & \vdots \\
            r_{a_2-1} & & s_{a_1-1} \\
            & r_{a_2} & \\
            & r_0 & \\
            r_1 & & t_1 \\
            r_2 & & t_2 \\
            \vdots & & \vdots \\
            r_{i-1} & & t_{a_1-1}
        \end{smallmatrix},
        \quad\quad\quad&M(r_i)_x := 
        \begin{cases}
            0 & \text{if $x=r_i$;} \\
            0 & \text{if $r_i=r_0$ and $x=t_j$ for any $j$;} \\
            0 & \text{if $r_i=r_{a_2}$ and $x=s_j$ for any $j$;} \\
            k & \text{otherwise.}
        \end{cases}
    \end{align*}
    Given any arrow $x \xrightarrow{\alpha} y$ in $Q_1$, the map $M(r_i)_x \xrightarrow{\varphi_\alpha} M(r_i)_y$ is defined by
    \begin{align*}
        \varphi_\alpha := 
        \begin{cases}
            1 & \text{if $M(r_i)_x = k = M(r_i)_y$;} \\
            0 & \text{otherwise.}
        \end{cases}
    \end{align*}
    Thus the support of $M(r_i)$ is given by removing the vertex $r_i$ from our quiver $Q$, together with vertices labeled $s_j$ when $i = a_2$, and respectively vertices labeled $t_j$ when $i = 0$. Equivalently, the support of $M(r_i)$ is given by removing the vertex $r_i$ from our quiver $Q$ and taking the component containing $r_{i+1}$.
\end{definition}

\begin{remark}\label{rem:M(ri)_min_path}
    For any minimal path $w$ between a vertex $x$ and a vertex $y$ of $Q$, let $\varphi_w$ be the composition of maps in the module $M(r_i)$ along the path $w$. If $M(r_i)_x = k$, then
    \begin{align*}
        \varphi_w = \begin{cases}
        1 & \text{if $r_{i}$ is not in $w$;} \\
        0 & \text{if $r_{i}$ is in $w$.}
        \end{cases}
    \end{align*}
\end{remark}

\begin{definition}\label{def:M(ti)}
    For all vertices labeled $t_i$, we define the module $M(t_i) = (M(t_i)_x, \varphi_\alpha)$ by the short exact sequence
    \begin{align*}
        0 \to M(t_i) \to I(r_{a_2}) \xrightarrow{f} I(s_i) \to 0,
    \end{align*}
    with $f$ given by the path $s_i \to s_{i+1} \to \cdots \to r_{a_2}$. For any vertex $x \in Q_0$, we have $M(t_i)_x = k$ if and only if $I(s_i)_x = 0$ and $I(r_{a_2})_x = k$. Also, $M(t_i)_x = 0$ otherwise. Therefore
    \begin{align*}
        &M(t_i) = 
        \begin{smallmatrix}
            r_{1} & & s_{i+1} \\
            r_{2} & & s_{i+2} \\
            \vdots & & \vdots \\
            r_{a_2-1} & & s_{a_1-1} \\
            & r_{a_2} &
        \end{smallmatrix}, 
        \quad\quad\quad&M(t_i)_x =
        \begin{cases}
            0 & \text{if $x=r_{0}$;} \\
            0 & \text{if $x=t_j$ for any $j$;} \\
            0 & \text{if $x=s_j$ for any $j \geq i$;} \\
            k & \text{otherwise.}
        \end{cases} 
    \end{align*}
    Given any arrow $x \xrightarrow{\alpha} y$ in $Q_1$, the map $M(t_i)_x \xrightarrow{\varphi_\alpha} M(t_i)_y$ is given by
    \begin{align*}
        \varphi_\alpha = 
        \begin{cases}
            1 & \text{if $M(t_i)_x = k = M(t_i)_y$;} \\
            0 & \text{otherwise.}
        \end{cases}
    \end{align*}
\end{definition}

\begin{definition}\label{def:M(si)}
    For all vertices labeled $s_i$, we define the module $M(s_i) = (M(s_i)_x, \varphi_\alpha)$ by the short exact sequence
    \begin{align*}
        0 \to P(t_i) \xrightarrow{g} P(r_0) \to M(s_i) \to 0,
    \end{align*}
    with $g$ given by the path $r_0 \to t_1 \to \cdots \to t_i$. For any vertex $x \in Q_0$, we have $M(s_i)_x = k$ if and only if $P(t_i)_x = 0$ and $P(r_0)_x = k$. Also, $M(s_i)_x = 0$ otherwise. Therefore
    \begin{align*}
        & M(s_i) = 
        \begin{smallmatrix}
            & r_0 & \\
            r_1 & & t_1 \\
            r_2 & & t_2 \\
            \vdots & & \vdots \\
            r_{a_2-1} & & t_{i-1}
        \end{smallmatrix}, 
        \quad\quad\quad&M(s_i)_x =
        \begin{cases}
            0 & \text{if $x=r_{a_2}$;} \\
            0 & \text{if $x=t_j$ for any $j \geq i$;} \\
            0 & \text{if $x=s_j$ for any $j$;} \\
            k & \text{otherwise.}
        \end{cases}
    \end{align*}
    Given any arrow $x \xrightarrow{\alpha} y$ in $Q_1$, the map $M(s_i)_x \xrightarrow{\varphi_\alpha} M(s_i)_y$ is given by
    \begin{align*}
        \varphi_\alpha = 
        \begin{cases}
            1 & \text{if $M(s_i)_x = k = M(s_i)_y$;} \\
            0 & \text{otherwise.}
        \end{cases}
    \end{align*}
\end{definition}

\begin{definition}\label{def:M}
    We define the module $T$ to be the sum of all modules defined in Definitions~\ref{def:M(ri)}, \ref{def:M(ti)}, and \ref{def:M(si)}. That is to say
    \begin{align*}
        T := \Bigg( \bigoplus_{i=0}^{a_2} M(r_i) \Bigg) \oplus \Bigg(\bigoplus_{i=1}^{a_1-1} M(s_i) \Bigg) \oplus \Bigg(\bigoplus_{i=1}^{a_1-1} M(t_i) \Bigg) = \bigoplus_{x \in Q_0} M(x).
    \end{align*}
\end{definition}

\section{Submodules of \texorpdfstring{$M(r_i)$}{M(ri)}.}\label{sec:submodules} 
As mentioned in the introduction, we want that \emph{each} indecomposable module $M(r_i)$ gives a parametrization of the Alexander polynomial of the link $L[a_1,a_2]$ via its submodules. In particular, since the Alexander polynomial has $a_1a_2+1$ terms, we should expect that the number of submodules of $M(r_i)$ also equals  $a_1a_2+1$ and does not depend on $i$. We prove this fact in this section. 

\begin{lemma}\label{lem:submod_supp_ra2}
    For every $i=0,1,\ldots,a_2$, the module
  $M(r_i)$ has exactly $a_1(a_2-i)$ submodules with support at $r_{a_2}$.
\end{lemma}

\begin{proof}
    If $i=a_2$, then there is nothing to show because $M(r_{a_2})$ is not supported at $r_{a_2}$. Otherwise, for each pair $(j,\ell)$ with  $i+1 \leq j \leq a_2$ and $1 \leq \ell \leq a_1$, we have a submodule
    \begin{align*}
        \begin{smallmatrix}
            r_j & & s_{\ell} \\
            \vdots & & \vdots \\
            r_{a_2-1} & & s_{a_1-1} \\
            & r_{a_2} & \\
            & r_0 & \\
            r_1 & & t_1 \\
            r_2 & & t_2 \\
            \vdots & & \vdots \\
            r_{i-1} & & t_{a_1-1}
        \end{smallmatrix},
    \end{align*}
    where we set $s_{a_1}=r_{a_2}$ if $\ell=a_1$.
    There are exactly $a_1(a_2-i)$ such submodules. 
\end{proof}

\begin{lemma}\label{lem:submod_supp_r0}
    For every $i=1, 2\ldots,a_2$, the module 
    $M(r_i)$ has exactly $1$ submodule that is supported at $r_0$ and not supported at $r_{a_2}$.
\end{lemma}

\begin{proof}
    The submodule is given by
    \begin{align*}
        \begin{smallmatrix}
            & r_0 & \\
            r_1 & & t_1 \\
            r_2 & & t_2 \\
            \vdots & & \vdots \\
            r_{i-1} & & t_{a_1-1}
        \end{smallmatrix}.
    \end{align*}
\end{proof}

\begin{lemma}\label{lem:submod_rest}
     For every $i=1,2,\ldots,a_2$, the module
$M(r_i)$ has exactly $(i\,a_1)$ submodules that are zero at  $r_0$ and $r_{a_2}$.
\end{lemma}

\begin{proof}
    For each pair $(j,\ell)$ such that  $1 \leq j \leq i$ and $1 \leq \ell \leq a_1$, we have a submodule
    \begin{align*}
        \begin{smallmatrix}
            r_j \\
            \vdots \\
            r_{i-1}
        \end{smallmatrix} \oplus
        \begin{smallmatrix}
            t_{\ell} \\
            \vdots \\
            t_{a_1-1}
        \end{smallmatrix}  =: L_1 \oplus L_2.
    \end{align*}
    Note that if $j = i$, then we take $L_1 = 0$. Similarly, if $\ell = a_1$, then we take $L_2 = 0$. There are exactly $i\,a_1$ such submodules.
    \end{proof}
    
    \begin{proposition}
        Let $r_i$ be any vertex on the cycle in $Q$. Then $M(r_i)$ has exactly $a_1a_2+1$ submodules.
    \end{proposition}
    
    \begin{proof}
         If $i=0$, then $M(r_0)$ has exactly $a_1a_2$ submodules supported at $r_{a_2}$, by \cref{lem:submod_supp_ra2}. Note that $M(r_0)=I(r_{a_2})$ is the injective at $r_{a_2}$. Thus the only other submodule of $M(r_0)$ is the zero module.
         
         Now suppose $i\ne 0$. Then  Lemmata~\ref{lem:submod_supp_ra2}--\ref{lem:submod_rest} yield an exhaustive list of the $a_1a_2+1$ submodules of $M(r_i)$.
%
    \end{proof}

\section{Computation of \texorpdfstring{$\tau T$}{TT}.} \label{sec:tau_M}
In this section, we compute the Auslander-Reiten (AR) translate of each indecomposable summand of the $A$-module $T$ defined in \cref{sec:def_of_M}. 

First, we compare the indecomposable summands of $T$ the indecomposible projectives and injectives of $A$, and we see the following.

\begin{remark}\label{rem:proj_inj}
    We have
\[    \begin{array}{rclcrclcrcl}
        M(r_{a_2-1}) &=& P(s_1), &\quad&
        M(r_{a_2}) &=& P(r_0),  &\quad&
        M(t_{a_1-1}) &=& P(r_1),\\[0.5pt]
        M(r_0) &=& I(r_{a_2}),  &\quad&
        M(r_1) &=& I(t_{a_1-1}), &\quad&
        M(s_1) &=& I(r_{a_2-1}).
    \end{array}\] 
\end{remark}
For the rest of this section, we use the notation  $L=(L_x, \varphi_\alpha)_{x\in Q_0,\alpha\in Q_1}$ for an $A$-module $L$.
\begin{lemma}\label{lem:tau_M(ri)}
      The modules $\tau M(r_{a_2})$ and $\tau M(r_{a_2-1})$ are zero, and for $i=0,1,\ldots, a_2-2$ we have
    \begin{align*}
        &\tau M(r_i) = 
        \begin{smallmatrix}
            s_2 & & r_{i+2} \\
            s_3 & & r_{i+3} \\
            \vdots & & \vdots \\
            s_{a_1-1} & & r_{a_2-1} \\
            & r_{a_2}&
        \end{smallmatrix},
        \quad\quad\quad &\tau M(r_i)_x =
        \begin{cases}
            0 &\text{if $x=r_j$ for any $0 \leq j \leq i+1$;} \\
            0 &\text{if $x=t_j$ for any $j$;} \\
            0 &\text{if $x=s_1$;} \\
            k &\text{otherwise.}
        \end{cases}
    \end{align*}
    Furthermore,  the projective dimension of $M(r_i)$ is $1$.
\end{lemma}

\begin{proof} Remark \ref{rem:proj_inj} implies that $\tau M(r_{a_2})$ and $\tau M(r_{a_2-1})$ are zero. Suppose now that $0\le i\le a_2-2$.
    We have the projective resolution
    \begin{align*}
        0 \to P(r_{a_2}) \xrightarrow{f} P(r_{i+1}) \oplus P(s_1) \to M(r_i) \to 0.
    \end{align*}
    This shows that $\operatorname{pd} M(r_i) = 1$. By applying the Nakayama functor $\nu$, we get the exact sequence
    \begin{align*}
        0 \to \tau M(r_i) \to I(r_{a_2}) \xrightarrow{\nu f} I(r_{i+1}) \oplus I(s_1).
    \end{align*} 
    The map $I(r_{a_2}) \to I(s_1)$ is nonzero only at the vertex $s_1$, so $\tau M(r_i)_{s_i} = 0$. The map $I(r_{a_2}) \to I(r_{i+1})$ is nonzero at the vertices $r_1, ..., r_{i+1}$, so $\tau M(r_i)$ is zero at these vertices. Also, both modules $I(r_{a_2})$ and $I(r_{i+1}) \oplus I(s_1)$ have no support at vertices labeled $t_j$, so $\tau M(r_i)$ is zero at these vertices. This completes the proof.
\end{proof}

\begin{lemma}\label{lem:tau_M(ti)}  The module   $\tau M(t_{a_1-1})$ is zero, and for $i=0,1,\ldots ,a_1-2$ we have
    \begin{align*}
        &\tau M(t_i) = 
        \begin{smallmatrix}
            s_{i+2} & & r_{2} \\
            s_{i+3} & & r_{3} \\
            \vdots & & \vdots \\
            s_{a_1-1} & & r_{a_2-1} \\
            & r_{a_2}&
        \end{smallmatrix},
        \quad\quad\quad &
        \tau M(t_i)_x =
        \begin{cases}
            0 &\text{if $x=r_0$ or $x=r_1$;} \\
            0 &\text{if $x=t_j$ for any $j$;} \\
            0 &\text{if $x=s_j$ for any $1 \leq j \leq i+1$;} \\
            k &\text{otherwise.}
        \end{cases}
    \end{align*}
    Furthermore, the projective dimension of $M(t_i) $ is $ 1$.
\end{lemma}

\begin{proof} Remark \ref{rem:proj_inj} implies that $\tau M(t_{a_1-1})$ is zero. Suppose now that $0\le i\le a_2-2$.
    We have the projective resolution
    \begin{align*}
        0 \to P(r_{a_2}) \xrightarrow{f} P(r_1) \oplus P(s_{i+1}) \to M(t_i) \to 0.
    \end{align*}
    This shows that $\operatorname{pd} M(t_i) = 1$. The rest of the proof is analogous to \cref{lem:tau_M(ri)}.
\end{proof}

\begin{lemma}\label{lem:tau_M(si)} 
          For  $i=1,2,\ldots, a_1 - 1$, we have     \begin{align*}
        &\tau M(s_i) = 
        \begin{smallmatrix}
            t_1 \\
            t_2 \\
            \vdots \\
            t_i
        \end{smallmatrix},
        \quad\quad\quad &
        \tau M(s_i)_x =
        \begin{cases}
            0 &\text{if $x = r_j$ for any $j$;} \\
            0 &\text{if $x = t_j$ for any $j > i$;} \\
            0 &\text{if $x = s_j$ for any $j$;} \\
            k &\text{otherwise.}
        \end{cases}
    \end{align*}
     Furthermore, the projective dimension of $M(s_i) $ is $ 1$.
\end{lemma}

\begin{proof}
    By definition of $M(s_i)$, we have the projective resolution
    \begin{align*}
        0 \to P(t_i) \xrightarrow{f} P(r_0) \to M(s_i) \to 0.
    \end{align*}
    This shows that $\operatorname{pd} M(s_i) = 1$. The rest of the proof is analogous to \cref{lem:tau_M(ri)}.
\end{proof}

\begin{proposition}\label{prop:pd_M}
    The module $T$ from \cref{def:M} has projective dimension $1$.
\end{proposition}

\begin{proof}
This follows immediately from  Lemmata \ref{lem:tau_M(ri)}--\ref{lem:tau_M(si)}.
\end{proof}


\section{\texorpdfstring{$T$}{T} is a tilting and a \texorpdfstring{$\tau$}{tau}-tilting module.}\label{sec:tilting_module}
In this section, we show that $T$ is a tilting $A$-module and a $\tau$-tilting $A$-module. We then show that $T$ induces a cluster-tilting object in the cluster category.

Recall that an $A$-module $T$ is called a \emph{tilting module} if $\textup{Ext}^1(T,T)=0$, $\operatorname{pd}T\le 1$, and the number $|T|$ of isoclasses of indecomposable summands of $T$ is equal to the number of vertices $|Q_0|$ of $Q$. Furthermore, $T$ is said to be \emph{$\tau$-tilting} in the sense of \cite{AIR} if $\operatorname{Hom}(T, \tau T) = 0$ and $|T|=|Q_0|$. Tilting modules and their endomorphism algebras play a central role in representation theory, see \cite{AsSiSk} for an introduction. 
 
Recall that an $A$-module $L = (L_x, \varphi_\alpha)$ is called \textit{thin} if all vector spaces $L_x$ are of dimension at most one.
        Let $Q(L)$ denote the subquiver of $Q$ containing all vertices $x \in Q_0$ such that $L_x \neq 0$ and all arrows $\alpha \in Q_1$ such that $\varphi_\alpha \neq 0$.   
          A thin module $L$ is indecomposable if and only if $Q(L)$ is connected.

\begin{remark}
 For every vertex $x$, the $A$-module $M(x)$ is thin and indecomposable.
\end{remark}
    
    \begin{lemma}\label{lem:zero_path}
        Let $L = (L_x, \varphi_\alpha)$ be a thin $A$-module and let $N = (N_x, \varphi_\alpha')$ be any $A$-module. Let $f \in \operatorname{Hom} (L, N)$ and let $w$ be a nonzero path in $Q(L)$ from a vertex $x$ to a vertex $y$. If $f_x = 0$, then $f_y = 0$.
    \end{lemma}
    
    \begin{proof}
        We have a commutative diagram:
            \[ \begin{tikzcd} L_x \arrow{r}{\varphi_w} \arrow[swap]{d}{f_x} & L_y \arrow{d}{f_y} \\ N_x \arrow{r}{\varphi'_w}& N_y \end{tikzcd} \]
    Since $L$ is thin, the map $\varphi_w$ is invertible, so $f_y = \varphi'_w f_x \varphi_w^{-1}$.
    \end{proof}
    
    \begin{lemma}\label{lem:M(r_i)_path}
        Let $L = M(r_i) = (M(r_i)_x, \varphi_\alpha)$ with $i \neq a_2$ and let $N = (N_x, \psi_\alpha)$ be indecomposable and thin. Suppose $\operatorname{Hom} (L,N) \neq 0$. Then $N_{s_1} \neq 0$ or there exists a vertex $r_j$ with $j \neq a_2-1, a_2$ such that $S(r_j)$ is a summand of the socle of $N$.
    \end{lemma}
    
    \begin{proof}
        Let $f \in \operatorname{Hom} (L,N)$ be a nonzero morphism. Suppose $N_{s_1} = 0$. Then \cref{lem:zero_path} implies $f_y = 0$ for all vertices $y$ that can be reached from $s_1$ by a path in $Q(L)$. In particular, $f_{r_{a_2}} = 0$.
        
        We have $\operatorname{top} L = S(r_{i+1}) \oplus S(s_1)$. Now consider the path $w : r_{i+1} \to \cdots \to r_{a_2}$. Note that $w$ runs through all vertices of $Q(L)$ which cannot be reached by a path from $s_1$ in $Q(L)$, plus the vertex $r_{a_2}$. Since $f$ is nonzero, it must be nonzero at some vertex on $w$. By \cref{lem:zero_path}, we see that $f$ must be nonzero at the vertex $r_{i+1}$. Now let $r_j$ be the last vertex in $w$ such that $f_{r_j} \neq 0$. Denote by $w'$ the subpath of $w$ from $r_{i+1}$ to $r_j$. Then we have a commutative diagram:
            \[ \begin{tikzcd} L_{r_{i+1}} \arrow{r}{\varphi_{w'}} \arrow[swap]{d}{0 \neq f_{r_{i+1}}} & L_{r_j} \arrow{r}{\varphi_\alpha} \arrow{d}{f_{r_j} \neq  0} & L_{r_{j+1}} \arrow{d}{f_{r_{j+1}} = 0} \\ N_{r_{i+1}} \arrow{r}{\psi_w} & N_{r_j} \arrow{r}{\psi_\alpha} & N_{r_{j+1}} \end{tikzcd} \]
        where $\alpha$ is the arrow $r_j \to r_{j+1}$. The left square implies $\psi_{w} \neq 0$ and the right square implies $\psi_\alpha = 0$. The arrow $\alpha$ is the only arrow whose source is $r_j$, since $1 \leq i+1 \leq j < a_2-1$. This implies that $S(r_j)$ is a summand of the socle of $N$.
    \end{proof}
    
    \begin{lemma}\label{lem:M(a2)M(si)_r0_support}
        Let $L = M(r_{a_2})$ or $M(s_i)$ with $i = 1,2, \cdots, a_1-1$ and let $N$ be indecomposable and thin. Suppose $\operatorname{Hom} (L,N) \neq 0$. Then $N_{r_0} \neq 0$.   
    \end{lemma}
    
    \begin{proof}
        We have $\operatorname{top} L = S(r_0)$ and every vertex in $Q(L)$ can be reached by a path from $r_0$ in $Q(L)$. By \cref{lem:zero_path}, if $f \in \operatorname{Hom} (L,N)$ is a nonzero morphism, then $f_{r_0} \neq 0$.
    \end{proof}
    
    \begin{lemma}\label{lem:M(ti)_r1_support}
        Let $L = M(t_i)$ with $i = 1,2,\cdots, a_1-1$ and let $N$ be indecomposable and thin. Suppose $\operatorname{Hom} (L,N) \neq 0$. Then $N_{r_1} \neq 0$ or there exists a vertex $s_j$ with $i+1 \leq j \leq a_1-1$ such that $S(s_j)$ is a summand of the socle of $N$.
    \end{lemma}
    
    \begin{proof}
        The proof is similar to the proof in \cref{lem:M(r_i)_path} with $r_1$ in the role of $s_1$ and the path $s_{i+1} \to \cdots \to s_{a_1-1} \to r_{a_2}$ in the role of $r_{i+1} \to \cdots \to r_{a_2-1} \to r_{a_2}$.
    \end{proof}
    
    \begin{lemma}\label{lem:ext_M}
        Let $L$ be an indecomposable summand of the module $T$ defined in \cref{def:M} and let $N$ be an indecomposable summand of $\tau T$. Then
        \begin{align*}
            \operatorname{Hom} (L,N) = 0.
        \end{align*}
        In particular,
        \begin{align*}
            \operatorname{Hom}(T,\tau T) = 0 \quad\quad \text{and} \quad\quad \operatorname{Ext}^1 (T,T) = 0.
        \end{align*}
    \end{lemma}
    
    \begin{proof}
        From our computations in \cref{sec:tau_M}, we see that each indecomposable summand $N$ of $\tau T$ is thin, and zero at the vertices $s_1$, $r_1$, and $r_0$. Using \cref{lem:M(a2)M(si)_r0_support} and $N_{r_0} = 0$, we see that $\operatorname{Hom} (L,N) =0$ for $L = M(r_{a_2})$ or $M(s_i)$. Using \cref{lem:zero_path} and $N_{s_1} = 0$, we see that $\operatorname{Hom}(L,N) = 0$ for $L = M(r_i)$, $i \neq a_2$ because none of the summands of $\tau T$ has $S(r_j)$ with $1 \leq j \leq a_2-1$ in its socle. Indeed, our computations in \cref{sec:tau_M} show that $\operatorname{soc} N$ can only contain $S(r_{a_2})$ and $S(t_i)$ as summands. Similarly, using \cref{lem:M(ti)_r1_support} and $N_{r_i} = 0$, we see that $\operatorname{Hom}(N,L)=0$ for $L = M(t_i)$ because none of the summands of $\tau T$ admit $S(s_j)$ as a summand of its socle. This shows $\operatorname{Hom}(L,N)=0$, and thus $  \operatorname{Hom}(T,\tau T) = 0. $

Because of the Auslander-Reiten formula \[\operatorname{Ext}^1(T,T) \cong D \operatorname{\overline{Hom}}(T, \tau T)\] this also implies the vanishing of $\operatorname{Ext}^1(T,T)$. 
       %
    \end{proof}
    
    \medskip
    We are now ready for the main results of this section.

    \begin{theorem}\label{thm:T-tilting}
        The module $T$ is a tilting $A$-module and a $\tau$-tilting $A$-module.
    \end{theorem}
    
    \begin{proof}
        The projective dimension of $T$ is $1$ by \cref{prop:pd_M} and $T$ is rigid by \cref{lem:ext_M}. In particular, we have shown that $\operatorname{Hom}(T, \tau T) = 0$. The result now follows since the number of non-isomorphic indecomposable summands of $T$ is equal to the number of vertices in $Q$.
    \end{proof}

\begin{theorem}\label{thm:T-cto}
    The module $T$ induces a cluster-tilting object in the cluster category.
\end{theorem}

\begin{proof}
    Since $T$ does not share any indecomposable summands with $A[1] = \bigoplus_{x \in Q_0} P(x)[1]$, we see that \cref{thm:T-tilting} and \cite[Theorem 4.1]{AIR}  complete the proof.
\end{proof}

\section{The Endomorphism algebra \texorpdfstring{$\operatorname{End} T$}{End T}.}\label{sec:end_M}

In this section, we study the endomorphism algebra of our $A$-module $T$ by computing the $\operatorname{Hom}$ spaces between its indecomposable summands.

\begin{lemma}\label{lem 6.1}
    Given two vertices  $r_i, r_j $ on the oriented cycle in $Q$, we have
    \begin{align*}
        \operatorname{Hom}(M(r_i),M(r_j)) = 
        \begin{cases}
            0 & \text{if $j = i+1$;} \\
            k & \text{if $j \neq i+1$.}
        \end{cases}
    \end{align*}
\end{lemma}

\begin{proof}
    Denote $M(r_i)$ by $(M(r_i)_x, \varphi_\alpha)$ and $M(r_j)$ by $(M(r_j)_x, \varphi'_\alpha)$. Furthermore, let $f \in \operatorname{Hom}(M(r_i),M(r_j))$ be any morphism. Then $f=(f_x)_{x\in Q_0}$ and each $f_x$ is given by the multiplication with a scalar. Let $\lambda\in k$ be the scalar corresponding to $ f_{r_{i+1}}$. Note that $\lambda$ may be zero. We complete the proof by showing that the choice of $\lambda$ completely determines the morphism $f$, that is, we prove that $f_x = 0$ or $f_x = \lambda$, for every vertex $x$. Also, we show that if $j \neq i+1$, then $\lambda$ can be nonzero. 
    
    We only need to consider vertices $x$ such that $M(r_i)_x$ and $M(r_j)_x$ are both nonzero; let $x \in Q_0$ be any vertex such that $M(r_i)_x = M(r_j)_x = k$. We consider two cases. 
    
    First, suppose there exists a nonzero path $w$ from $r_{i+1}$ to $x$. Let $\varphi_w$ be the composition of maps from $r_{i+1}$ to $x$ in our module $M(r_i)$ along the path $w$, and let $\varphi'_w$ be the composition of maps from $r_{i+1}$ to $x$ in our module $M(r_j)$ along the path $w$. Note that by our assumption that $M(r_i)_x \neq 0$, we have $\varphi_w = 1$. On the other hand, $\varphi'_w = 0$ or $\varphi'_w = 1$. Since $f$ is a morphism, we have a commutative diagram:
        \[ \begin{tikzcd} M(r_i)_{r_i+1} \arrow{r}{\varphi_w} \arrow[swap]{d}{\lambda} & M(r_i)_x \arrow{d}{f_x} \\ M(r_j)_{r_i+1} \arrow{r}{\varphi'_w}& M(r_j)_x \end{tikzcd} \]
    Hence $f_x = \lambda \varphi'_w$, so $f_x = 0$ or $f_x = \lambda$.
    
    Second, suppose there does not exist a nonzero path from $r_{i+1}$ to $x$. Then $x=s_\ell$ for some $\ell$, or $x=r_i$, or $i=0$ and $x=r_0,t_1,\ldots,t_{a_1-1}.$ 
    If $x=r_i$ then $f_x=0$, since $M(r_i)_{x}=0$. Similarly, if $i=0$ and $x$ is one of $r_0,t_1,\ldots,t_{a_1-1}$ then $f_x=0$, because $M(r_0)_x$ is zero.  It remains the case where
      $x=s_\ell$. let $w$ be the nonzero path from $s_\ell$ to $r_{a_2}$. Let $\varphi_w$ be the composition of maps from $s_\ell$ to $r_{a_2}$ in $M(r_i)$ along the path $w$ and let $\varphi'_w$ be the composition of maps from $s_\ell$ to $r_{a_2}$ in $M(r_j)$ along the path $w$. In this case, $\varphi_w = 1$ and $\varphi'_w = 1$. Since $f$ is a morphism, we get a commutative diagram:
        \[ \begin{tikzcd} M(r_i)_{s_\ell} \arrow{r}{\varphi_w} \arrow[swap]{d}{f_{s_\ell}} & M(r_i)_{r_{a_2}} \arrow{d}{f_{r_{a_2}}} \\ M(r_j)_{s_\ell} \arrow{r}{\varphi'_w}& M(r_j)_{r_{a_2}} \end{tikzcd} \]
    Hence $f_x = f_{r_{a_2}}$, and since we have already shown that $f_{r_{a_2}} = 0$ or $f_{r_{a_2}} = \lambda$, we conclude that $\lambda$ completely determines $f$.
    
    To complete the proof, suppose first $j = i+1$. Then $\lambda = f_{r_{i+1}} = 0$, since $M(r_j)_{i+1} = 0$, and thus $\operatorname{Hom} (M(r_i), M(r_{j})) = 0$. Otherwise, $j \neq i+1$ and $\lambda\in k$ can be chosen arbitrarily. Thus $\operatorname{Hom}(M(r_i), M(r_{j}))= k$.
\end{proof}

\begin{lemma}\label{lem:hom(M(ti) M(tj))}
    Given two vertices labeled $t_i, t_j \in Q_0$, we have
    \begin{align*}
        \operatorname{Hom}(M(t_i),M(t_j)) =
        \begin{cases}
            0 & \text{if $i < j$;} \\
            k & \text{if $i \geq j$.}
        \end{cases}
    \end{align*}
\end{lemma}

\begin{proof}
    This follows from the observation that if $i \geq j$, then $M(t_i)$ is a submodule of $M(t_j)$, so $\operatorname{Hom}(M(t_i),M(t_j))$ is generated by the inclusion map.
\end{proof}

\begin{lemma}\label{lem 6.3}
    Given two vertices labeled $s_i, s_j \in Q_0$, we have
    \begin{align*}
        \operatorname{Hom}(M(s_i),M(s_j)) =
        \begin{cases}
            0 & \text{if $i < j$;} \\
            k & \text{if $i \geq j$.}
        \end{cases}
    \end{align*}
\end{lemma}

\begin{proof}
    Since $Q$ is isomorphic to $Q^{op}$, we see that this is the dual argument to \cref{lem:hom(M(ti) M(tj))}.
\end{proof}

\begin{lemma}\label{lem:hom(-M(si))}
    Let $x \in Q_0$ be any vertex labeled $r_j$ or $t_j$. Given any vertex labeled $s_i \in Q_0$, we have
    \begin{align*}
        \operatorname{Hom}(M(x), M(s_i)) = \operatorname{Hom}(M(x), M(r_{a_2})).
    \end{align*}
\end{lemma}

\begin{proof}
    From \cref{def:M(si)}, we have a short exact sequence
    \begin{align*}
        0 \to P(t_i) \to M(r_{a_2}) \to M(s_i) \to 0.
    \end{align*}
    We apply the functor $\operatorname{Hom}(M(x), -)$ to the sequence to get the exact sequence
    \begin{align*}
        0 &\to \operatorname{Hom}(M(x), P(t_i)) \to \operatorname{Hom}(M(x), M(r_{a_2})) \to \operatorname{Hom}(M(x), M(s_i))\\ &\to \operatorname{Ext}^1(M(x), P(t_i)).
    \end{align*}
    We complete the proof by showing that $\operatorname{Hom}(M(x), P(t_i))=0$ and $\operatorname{Ext}^1(M(x), P(t_i)) = 0$.
    
    First, let $f \in \operatorname{Hom}(M(x), P(t_i))$. By \cref{def:M(ri)} and \cref{def:M(ti)}, we have
    \begin{align*}
        \operatorname{top} M(x) = \begin{cases}
            S(s_1) \oplus S(r_{j+1}) & \text{if $x=r_j$;} \\
            S(r_1) \oplus S(s_{j+1}) & \text{if $x=t_j$.}
        \end{cases}
    \end{align*}
    In both cases, there is a nonzero path from a vertex $y$ in the top of $M(x)$ to $t_i$, and $P(t_i)$ is not supported on $y$. \cref{lem:zero_path} implies that $f_{t_i} = 0$, hence $f = 0$.
    
    Second, we see that $\operatorname{Hom} (P(t_i), \tau M(x)) = 0$ by \cref{lem:ext_M}. Hence by the AR formula, $\operatorname{Ext}^1(M(x), P(t_i)) = D \overline{\operatorname{Hom}} (P(t_i), \tau M(x)) = 0$.
\end{proof}

\begin{lemma}\label{lem 6.5}
    Let $x \in Q_0$ be any vertex labeled $r_j$ or $s_j$. Given any vertex labeled $t_i \in Q_0$, we have
    \begin{align*}
        \operatorname{Hom}(M(t_i), M(x)) \cong \operatorname{Hom}(M(t_i), M(r_{0})).
    \end{align*}
\end{lemma}

\begin{proof}
    Since $Q$ is isomorphic to $Q^{op}$, the argument is dual to \cref{lem:hom(-M(si))}.
\end{proof}

\begin{lemma}\label{lem:Hom(M(ri) M(tj))}
    Given two vertices labeled $r_i, t_j \in Q_0$, we have
    \begin{align*}
        \operatorname{Hom}(M(r_i), M(t_j)) = 0.
    \end{align*}
\end{lemma}

\begin{proof}
    This follows from \cref{lem:zero_path} and \cref{lem:M(a2)M(si)_r0_support}.
\end{proof}

\begin{lemma}\label{lem 6.7}
    Given two vertices labeled $r_i, s_j \in Q_0$, we have
    \begin{align*}
        \operatorname{Hom}(M(s_j), M(r_i)) = 0.
    \end{align*}
\end{lemma}

\begin{proof}
    The argument is dual to \cref{lem:Hom(M(ri) M(tj))}.
\end{proof}

\begin{lemma}\label{lem 6.8}
    Given two vertices labeled $s_i, t_j \in Q_0$, we have
    \begin{align*}
        \operatorname{Hom}(M(s_i), M(t_j)) = 0.
    \end{align*}
\end{lemma}

\begin{proof}
    This follows from the fact that the top of $M(s_i)$ is $S(r_0)$ and the vertex $r_0$ does not lie in the support of $M(t_j)$.
\end{proof}

\medskip

The results of this subsection can be combined in the following theorem. Recall that $A$ is isomorphic to its dual.
\begin{theorem}\label{thm:end_T}
The endomorphism algebra of $T$ is isomorphic to (the dual of) $A$,
\[\operatorname{End}_A T\cong A.\]
Furthermore, the mapping $x \mapsto M(x)$ induces an isomorphism of quivers $Q_A^{\operatorname{op}} \to Q_{\operatorname{End}_A T}$. In particular, the quiver $Q_{\operatorname{End}_A T}$ of $\operatorname{End}_A T$ is given by
    \begin{center}
    \begin{tikzpicture}[<-,scale=.7] 
        \node (r0) at (60:2.5cm) {$M(r_0)$}; \node (ra2) at (120:2.5cm) {$M(r_{a_2})$};
        \node (ra2-1) at (180:2.5cm) {$M(r_{a_2-1})$};
        \node (r1) at (0:2.5cm) {$M(r_1)$};
        
        \node (t1) at ($(r0) + (2.5,0)$) {$M(t_1)$};
        \node (t2) at ($(t1) + (1.7,0)$) {};
        \node (ta1-2) at ($(t2) + (1.7,0)$) {};
        \node (ta1-1) at ($(ta1-2) + (2.1,0)$) {$M(t_{a_1-1})$};
        
        \node (sa1-1) at ($(ra2) + (-3,0)$) {$M(s_{a_1-1})$};
        \node (sa1-2) at ($(sa1-1) + (-2.1,0)$) {};
        \node (s2) at ($(sa1-2) + (-1.7,0)$) {};
        \node (s1) at ($(s2) + (-1.7,0)$) {$M(s_{1})$};
        
        \draw (s1) -- (s2);
        \draw[thick, loosely dotted, -] (s2) -- (sa1-2);
        \draw (sa1-2) -- (sa1-1);
        \draw (sa1-1) -- (ra2);
        
        \draw (r0) -- (t1);
        \draw (t1) -- (t2);
        \draw[thick, loosely dotted, -] (t2) -- (ta1-2);
        \draw (ta1-2) -- (ta1-1);
        
        \draw (165:2.5cm) arc (165:135:2.5cm);
        \draw (45:2.5cm) arc (45:15:2.5cm);
        \draw (ra2) -- (r0);
        \draw (225:2.5cm) arc (225:195:2.5cm);
        \draw (-15:2.5cm) arc (-15:-45:2.5cm);
        \draw[thick, loosely dotted, -] (-55:2.5cm) arc (-55:-125:2.5cm);
    
    \end{tikzpicture}
    \end{center} 
\end{theorem}
\begin{proof}

By Lemma~\ref{lem 6.1}, the full subquiver with vertices $M(r_i)$ is an oriented cycle of length $a_2+1$ in which every subpath of length $a_2$ is zero. Lemmata \ref{lem:hom(M(ti) M(tj))} and \ref{lem 6.3} imply that the full subquivers with vertices $M(s_i)$ and $M(t_i)$, respectively, are equioriented of  type $\mathbb{A}$ and there are no relations on these branches.
Lemmata~\ref{lem:hom(-M(si))} and \ref{lem 6.5} imply that there  is an arrow $M(r_{a_2})\to M(s_{a_1-1})$ and an arrow $M(t_1)\to M(r_0)$ and that there are no relations between the cycle and the two branches. 
The Lemmata~\ref{lem:Hom(M(ri) M(tj))} -- \ref{lem 6.8}
show that are no other arrows. This shows that the quiver of $\textup{End}_AT$ is the one in the theorem and its relations are those coming from the potential.
\end{proof}

\section{Mutation and the cluster algebra \texorpdfstring{$\mathcal{A}(Q)$}{A(Q)}}\label{sec:mutation}

In \cref{thm:T-cto}, we have shown that $T$ induces a cluster-tilting object in the cluster category $\mathcal{C}_Q$ introduced in \cite{BMRRT}. From \cref{thm:end_T}, the quiver given by $T$ is the opposite quiver $Q^{op}$, so using a result in \cite{ASS} we see that the correspondence $T \mapsto A[1]$ induces a cluster automorphism $\sigma$.

In this section, we define a mutation sequence $\mu$ on the quiver $Q$ and show that $Q$ is of acyclic type. The cluster automorphism $\sigma$ is given by $\mu$, which we prove by showing that $\mu$ sends $T$ to $A[1] = \oplus_{x \in Q_0} P(x)[1]$ in the cluster category. Furthermore, we show that $\sigma$ has order $2$.

\begin{definition}
    We define the mutation sequence $\mu$ on the quiver $Q$ by
    \begin{align*}
        \mu := \mu_{\mkern-1.5mu R}^{-1} \circ \mu_S \circ \mu_T \circ \mu_{\mkern-1.5mu R},
    \end{align*}
    where
    \begin{align*}
        &\mu_{\mkern-1.5mu R} = \mu_{r_{a_2}-1} \circ \cdots \circ \mu_{r_2} \circ \mu_{r_1}, \\
        &\mu_S = \mu_{s_1}\mu_{s_2}\mu_{s_3} \cdots \mu_{r_{a_2}} \circ \cdots \circ \mu_{s_1}\mu_{s_2}\mu_{s_3} \circ \mu_{s_1}\mu_{s_2} \circ \mu_{s_1}, \\
        &\mu_T = \mu_{t_{a_1-1}}\mu_{t_{a_1-2}}\mu_{t_{a_1-3}} \cdots \mu_{r_0} \circ \cdots \circ \mu_{t_{a_1-1}}\mu_{t_{a_1-2}}\mu_{t_{a_1-3}} \circ \mu_{t_{a_1-1}}\mu_{t_{a_1-2}} \circ \mu_{t_{a_1-1}}.
    \end{align*}
\end{definition}

We see the effect of these mutation sequences on the quiver illustrated in Figures~\ref{fig:mutation_R}-\ref{fig:mu_Q}. The quiver $\mu_{\mkern-1.5mu R}\, Q$ is seen in the left picture of \cref{fig:mutation_ST}.

\begin{figure}[!htb]
    \centering
    \begin{tikzpicture}[->,scale=.4,,baseline={(0,-0.17)},every node/.style={scale=0.7}] 
        \node (r0) at (60:2cm) {$r_0$}; \node (ra2) at (120:2cm) {$r_{a_2}$};
        \node (r1) at (0:2cm) {$r_1$};

        \draw (165:2cm) arc (165:135:2cm);
        \draw (45:2cm) arc (45:15:2cm);
        \draw (105:2cm) arc (105:75:2cm);
        \draw (-15:2cm) arc (-15:-45:2cm);
        \draw[thick, loosely dotted, -] (-55:2cm) arc (-55:-195:2cm);
    \end{tikzpicture}
    $\overset{\mu_{r_1}}{\leadsto}$
    \begin{tikzpicture}[->,scale=.4,baseline={(0,-0.17)},every node/.style={scale=0.7}] 
        \node (r0) at (60:2cm) {$r_0$}; \node (ra2) at (120:2cm) {$r_{a_2}$};
        \node (r1) at (0:2cm) {$r_1$};
        \node (r2) at (-50:2cm) {$r_2$};

        \draw (165:2cm) arc (165:135:2cm);
        \draw (15:2cm) arc (15:45:2cm);
        \draw (105:2cm) arc (105:75:2cm);
        \draw (r0) -- (r2);
        \draw (-35:2cm) arc (-35:-15:2cm);
        \draw (-65:2cm) arc (-65:-85:2cm);
        \draw[thick, loosely dotted, -] (-90:2cm) arc (-90:-195:2cm);
    \end{tikzpicture}
    $\overset{\mu_{r_2}}{\leadsto}$
    \begin{tikzpicture}[->,scale=.4, baseline={(0,-0.17)},every node/.style={scale=0.7}] 
        \node (r0) at (60:2cm) {$r_0$}; \node (ra2) at (120:2cm) {$r_{a_2}$};

        \node (r1) at (0:2cm) {$r_1$};
        \node (r2) at (-50:2cm) {$r_2$};
        \node (r3) at (-90:2cm) {$r_3$};
        \draw (165:2cm) arc (165:135:2cm);
        \draw (105:2cm) arc (105:75:2cm);
        \draw (r2) -- (r0);
        \draw (r0) -- (r3);
        \draw (-15:2cm) arc (-15:-35:2cm);
        \draw (-75:2cm) arc (-75:-65:2cm);
        \draw (-105:2cm) arc (-105:-120:2cm);
        \draw[thick, loosely dotted, -] (-125:2cm) arc (-125:-190:2cm);
    \end{tikzpicture}
    $\overset{\mu_{r_3}}{\leadsto} \cdots \overset{\mu_{r_{i-1}}}{\leadsto}$
    \begin{tikzpicture}[->,scale=.4, ,baseline={(0,-0.17)},every node/.style={scale=0.7}] 
        \node (r0) at (60:2cm) {$r_0$}; \node (ra2) at (120:2cm) {$r_{a_2}$};

        \node (r1) at (0:2cm) {$r_1$};
        \node (ri-1) at (-95:2cm) {$r_{i-1}$};
        \node(ri) at (-145:2cm) {$r_i$};
        \draw (155:2cm) arc (155:135:2cm);
        \draw (105:2cm) arc (105:75:2cm);
        \draw (ri-1) -- (r0);
        \draw (r0) -- (ri);
        \draw (-15:2cm) arc (-15:-30:2cm);
        \draw (-60:2cm) arc (-60:-75:2cm);
        \draw (-130:2cm) arc (-130:-117:2cm);
        \draw (-155:2cm) arc (-155:-175:2cm);
        \draw[thick, loosely dotted, -] (-185:2cm) arc (-185:-200:2cm);
        \draw[thick, loosely dotted, -] (-35:2cm) arc (-35:-55:2cm);
    \end{tikzpicture}
    $\overset{\mu_{r_i}}{\leadsto} \cdots \overset{\mu_{r_{a_2-1}}}{\leadsto}$
    \begin{tikzpicture}[->,scale=.4,baseline={(0,-0.15)},every node/.style={scale=0.7}] 
        \node (r0) at (60:2cm) {$r_0$}; \node (ra2) at (120:2cm) {$r_{a_2}$};

        \node (r1) at (0:2cm) {$r_1$};
        \node (ra2-1) at (170:2cm) {$r_{a_2-1}$};
        
        \draw (ra2-1) -- (r0);
        \draw (135:2cm) arc (135:155:2cm);
        \draw (-15:2cm) arc (-15:-35:2cm);
        \draw (-155:2cm) arc (-155:-175:2cm);
        \draw[thick, loosely dotted, -] (-45:2cm) arc (-45:-145:2cm);
    \end{tikzpicture}
    \caption{The mutation sequence $\mu_{\mkern-1.5mu R}$ on the cycle $r_0 \to r_1 \to \cdots \to r_{a_2} \to r_0$.}
    \label{fig:mutation_R}
\end{figure}
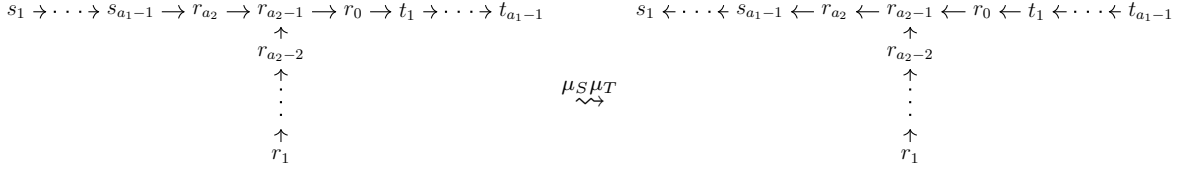
\begin{figure}[!htb]
    \centering
    \begin{tikzpicture}[->,scale=.4,baseline={(0,-1.3)},every node/.style={scale=0.7}] 
        \node (s_1) at (0,0) {$s_1$};
        \node (s_a1-1) at (3.8,0) {$s_{a_1-1}$};
        \node (r_a2) at (6.3,0) {$r_{a_2}$};
        \node (r_a2-1) at (8.8,0) {$r_{a_2-1}$};
        \node (r_0) at (11.2,0) {$r_0$};
        \node (t_1) at (13,0) {$t_1$};
        \node (t_a1-1) at (16.8,0) {$t_{a_1-1}$};
        
        \node (r_a2-2) at (8.8,-1.4) {$r_{a_2-2}$};
        \node (r_1) at (8.8, -4.8) {$r_1$};
        
        \draw (s_a1-1) -- (r_a2);
        \draw (r_a2) -- (r_a2-1);
        \draw (r_a2-1) -- (r_0);
        \draw (r_0) -- (t_1);
        
        \draw[->] (s_1) -- (1,0);
        \draw[thick, loosely dotted, -] (1.25,0) -- (2.45,0);
        \draw[->] (2.4,0) -- (s_a1-1);
        
        \draw[->] (t_1) -- (14,0);
        \draw[thick, loosely dotted, -] (14.25,0) -- (15.45,0);
        \draw[->] (15.4,0) -- (t_a1-1);
        
        \draw[<-] (r_a2-1) -- (r_a2-2);
        \draw[<-] (r_a2-2) -- (8.8, -2.3);
        \draw[thick, loosely dotted, -] (8.8,-2.55) -- (8.8,-3.75);
        \draw[<-] (8.8,-3.9) -- (r_1);
    \end{tikzpicture}
    $\overset{{\mu_S\mu_T}}{\leadsto}$
    \begin{tikzpicture}[->,scale=.4,baseline={(0,-1.3)},every node/.style={scale=0.7}] 
        \node (s_1) at (0,0) {$s_1$};
        \node (s_a1-1) at (3.8,0) {$s_{a_1-1}$};
        \node (r_a2) at (6.3,0) {$r_{a_2}$};
        \node (r_a2-1) at (8.8,0) {$r_{a_2-1}$};
        \node (r_0) at (11.2,0) {$r_0$};
        \node (t_1) at (13,0) {$t_1$};
        \node (t_a1-1) at (16.8,0) {$t_{a_1-1}$};
        
        \node (r_a2-2) at (8.8,-1.4) {$r_{a_2-2}$};
        \node (r_1) at (8.8, -4.8) {$r_1$};
        
        \draw[<-] (s_a1-1) -- (r_a2);
        \draw[<-] (r_a2) -- (r_a2-1);
        \draw[<-] (r_a2-1) -- (r_0);
        \draw[<-] (r_0) -- (t_1);
        
        \draw[<-] (s_1) -- (1,0);
        \draw[thick, loosely dotted, -] (1.25,0) -- (2.45,0);
        \draw[<-] (2.4,0) -- (s_a1-1);
        
        \draw[<-] (t_1) -- (14,0);
        \draw[thick, loosely dotted, -] (14.25,0) -- (15.45,0);
        \draw[<-] (15.4,0) -- (t_a1-1);
        
        \draw[<-] (r_a2-1) -- (r_a2-2);
        \draw[<-] (r_a2-2) -- (8.8, -2.3);
        \draw[thick, loosely dotted, -] (8.8,-2.55) -- (8.8,-3.75);
        \draw[<-] (8.8,-3.9) -- (r_1);
    \end{tikzpicture}
    \caption{The mutation sequence $\mu_S\mu_T$ on the quiver $\mu_{\mkern-1.5mu R}\, Q$.}
    \label{fig:mutation_ST}
\end{figure}
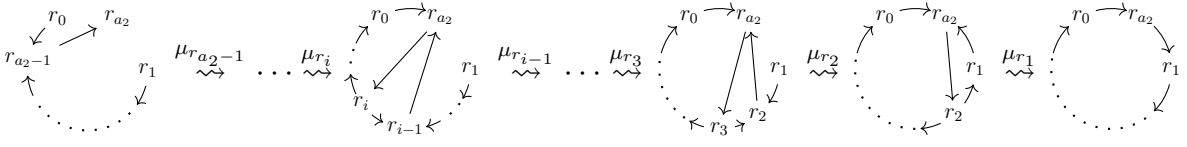
\begin{figure}[!htb]
    \centering
        \begin{tikzpicture}[->,scale=.4,baseline={(0,-0.15)},every node/.style={scale=0.7}] 
        \node (r0) at (120:2cm) {$r_0$}; \node (ra2) at (60:2cm) {$r_{a_2}$};

        \node (r1) at (0:2cm) {$r_1$};
        \node (ra2-1) at (170:2cm) {$r_{a_2-1}$};
        
        \draw (ra2-1) -- (ra2);
        \draw (135:2cm) arc (135:155:2cm);
        \draw (-15:2cm) arc (-15:-35:2cm);
        \draw (-155:2cm) arc (-155:-175:2cm);
        \draw[thick, loosely dotted, -] (-45:2cm) arc (-45:-145:2cm);
    \end{tikzpicture}
    $\overset{\mu_{r_{a_2-1}}}{\leadsto} \cdots \overset{\mu_{r_{i}}}{\leadsto}$
    \begin{tikzpicture}[->,scale=.4, ,baseline={(0,-0.17)},every node/.style={scale=0.7}] 
        \node (r0) at (120:2cm) {$r_0$}; \node (ra2) at (60:2cm) {$r_{a_2}$};

        \node (r1) at (0:2cm) {$r_1$};
        \node (ri-1) at (-95:2cm) {$r_{i-1}$};
        \node(ri) at (-145:2cm) {$r_i$};
        \draw (155:2cm) arc (155:135:2cm);
        \draw (105:2cm) arc (105:75:2cm);
        \draw (ri-1) -- (ra2);
        \draw (ra2) -- (ri);
        \draw (-15:2cm) arc (-15:-30:2cm);
        \draw (-60:2cm) arc (-60:-75:2cm);
        \draw (-130:2cm) arc (-130:-117:2cm);
        \draw (-155:2cm) arc (-155:-175:2cm);
        \draw[thick, loosely dotted, -] (-185:2cm) arc (-185:-200:2cm);
        \draw[thick, loosely dotted, -] (-35:2cm) arc (-35:-55:2cm);
    \end{tikzpicture}
    $\overset{\mu_{r_{i-1}}}{\leadsto} \cdots \overset{\mu_{r_{3}}}{\leadsto}$
    \begin{tikzpicture}[->,scale=.4, baseline={(0,-0.17)},every node/.style={scale=0.7}] 
        \node (r0) at (120:2cm) {$r_0$}; \node (ra2) at (60:2cm) {$r_{a_2}$};

        \node (r1) at (0:2cm) {$r_1$};
        \node (r2) at (-50:2cm) {$r_2$};
        \node (r3) at (-90:2cm) {$r_3$};
        \draw (165:2cm) arc (165:135:2cm);
        \draw (105:2cm) arc (105:75:2cm);
        \draw (r2) -- (ra2);
        \draw (ra2) -- (r3);
        \draw (-15:2cm) arc (-15:-35:2cm);
        \draw (-75:2cm) arc (-75:-65:2cm);
        \draw (-105:2cm) arc (-105:-120:2cm);
        \draw[thick, loosely dotted, -] (-125:2cm) arc (-125:-190:2cm);
    \end{tikzpicture}
    $\overset{\mu_{r_{2}}}{\leadsto}$
    \begin{tikzpicture}[->,scale=.4,baseline={(0,-0.17)},every node/.style={scale=0.7}] 
        \node (r0) at (120:2cm) {$r_0$}; \node (ra2) at (60:2cm) {$r_{a_2}$};
        \node (r1) at (0:2cm) {$r_1$};
        \node (r2) at (-50:2cm) {$r_2$};

        \draw (165:2cm) arc (165:135:2cm);
        \draw (15:2cm) arc (15:45:2cm);
        \draw (105:2cm) arc (105:75:2cm);
        \draw (ra2) -- (r2);
        \draw (-35:2cm) arc (-35:-15:2cm);
        \draw (-65:2cm) arc (-65:-85:2cm);
        \draw[thick, loosely dotted, -] (-90:2cm) arc (-90:-195:2cm);
    \end{tikzpicture}
    $\overset{\mu_{r_{1}}}{\leadsto}$
    \begin{tikzpicture}[->,scale=.4,,baseline={(0,-0.17)},every node/.style={scale=0.7}] 
        \node (r0) at (120:2cm) {$r_0$}; \node (ra2) at (60:2cm) {$r_{a_2}$};
        \node (r1) at (0:2cm) {$r_1$};

        \draw (165:2cm) arc (165:135:2cm);
        \draw (45:2cm) arc (45:15:2cm);
        \draw (105:2cm) arc (105:75:2cm);
        \draw (-15:2cm) arc (-15:-45:2cm);
        \draw[thick, loosely dotted, -] (-55:2cm) arc (-55:-195:2cm);
    \end{tikzpicture}
    \caption{The mutation sequence $\mu_{\mkern-1.5mu R}$ on the cycle $r_1 \to r_2 \to \cdots \to  r_{a_2-1} \to r_0 \to r_{a_2} \to r_1$, with $\mu_{\mkern-1.5mu R}^{-1}$ shown.}
    \label{fig:mutation_R_inverse}
\end{figure}
\begin{figure}[!htb]
    \centering
    \begin{tikzpicture}[->,scale=.4,every node/.style={scale=0.7}, baseline={(0,-0.2)}] 
        \node (r0) at (60:2cm) {$r_0$}; \node (ra2) at (120:2cm) {$r_{a_2}$};
        \node (ra2-1) at (180:2cm) {$r_{a_2-1}$};
        \node (r1) at (0:2cm) {$r_1$};
        
        \node (t1) at ($(r0) + (2,0)$) {$t_1$};
        \node (t2) at ($(t1) + (1.2,0)$) {};
        \node (ta1-2) at ($(t2) + (1.2,0)$) {};
        \node (ta1-1) at ($(ta1-2) + (1.6,0)$) {$t_{a_1-1}$};
        
        \node (sa1-1) at ($(ra2) + (-2.5,0)$) {$s_{a_1-1}$};
        \node (sa1-2) at ($(sa1-1) + (-1.6,0)$) {};
        \node (s2) at ($(sa1-2) + (-1.2,0)$) {};
        \node (s1) at ($(s2) + (-1.2,0)$) {$s_{1}$};
        
        \draw (s1) -- (s2);
        \draw[thick, loosely dotted, -] (s2) -- (sa1-2);
        \draw (sa1-2) -- (sa1-1);
        \draw (sa1-1) -- (ra2);
        
        \draw (r0) -- (t1);
        \draw (t1) -- (t2);
        \draw[thick, loosely dotted, -] (t2) -- (ta1-2);
        \draw (ta1-2) -- (ta1-1);
        
        \draw (165:2cm) arc (165:135:2cm);
        \draw (45:2cm) arc (45:15:2cm);
        \draw (105:2cm) arc (105:75:2cm);
        \draw (225:2cm) arc (225:195:2cm);
        \draw (-15:2cm) arc (-15:-45:2cm);
        \draw[thick, loosely dotted, -] (-55:2cm) arc (-55:-125:2cm);
    
    \end{tikzpicture}
    $\overset{\mu}{\leadsto}$
    \begin{tikzpicture}[<-,scale=.4,every node/.style={scale=0.7}, baseline={(0,-0.2)}] 
        \node (r0) at (60:2cm) {$r_0$}; \node (ra2) at (120:2cm) {$r_{a_2}$};
        \node (ra2-1) at (180:2cm) {$r_1$};
        \node (r1) at (0:2cm) {$r_{a_2-1}$};
        
        \node (t1) at ($(r0) + (2,0)$) {$t_1$};
        \node (t2) at ($(t1) + (1.2,0)$) {};
        \node (ta1-2) at ($(t2) + (1.2,0)$) {};
        \node (ta1-1) at ($(ta1-2) + (1.6,0)$) {$t_{a_1-1}$};
        
        \node (sa1-1) at ($(ra2) + (-2.5,0)$) {$s_{a_1-1}$};
        \node (sa1-2) at ($(sa1-1) + (-1.6,0)$) {};
        \node (s2) at ($(sa1-2) + (-1.2,0)$) {};
        \node (s1) at ($(s2) + (-1.2,0)$) {$s_{1}$};
        
        \draw (s1) -- (s2);
        \draw[thick, loosely dotted, -] (s2) -- (sa1-2);
        \draw (sa1-2) -- (sa1-1);
        \draw (sa1-1) -- (ra2);
        
        \draw (r0) -- (t1);
        \draw (t1) -- (t2);
        \draw[thick, loosely dotted, -] (t2) -- (ta1-2);
        \draw (ta1-2) -- (ta1-1);
        
        \draw (165:2cm) arc (165:135:2cm);
        \draw (45:2cm) arc (45:15:2cm);
        \draw (105:2cm) arc (105:75:2cm);
        \draw (225:2cm) arc (225:195:2cm);
        \draw (-15:2cm) arc (-15:-45:2cm);
        \draw[thick, loosely dotted, -] (-55:2cm) arc (-55:-125:2cm);
    
    \end{tikzpicture}
    \caption{The mutation sequence $\mu = \mu_{\mkern-1.5mu R}^{-1}\mu_T \mu_S \mu_{\mkern-1.5mu R}$ on the quiver $Q$.}
    \label{fig:mu_Q}
\end{figure}
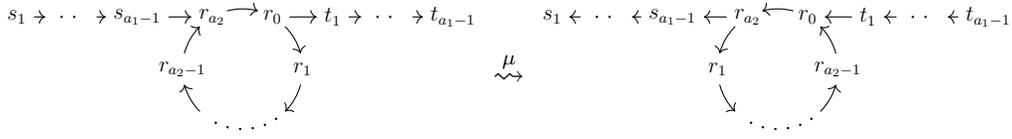

\begin{theorem}
    The cluster algebra $\mathcal{A}(Q)$ is of acyclic type $\mu_{\mkern-1.5mu R}\, Q$. Moreover, the cases where it is of finite or tame type are the following.
 \[ \begin{array}{ll}
 \mathbb{A}_{2a_1+1} & \textup{if $a_2=2$};\\
 \mathbb{D}_{a_2+1} & \textup{if $a_1=1$};\\ 
 \mathbb{E}_{6} & \textup{if $(a_1,a_2)=(2,3)$};\\
 \widetilde{\mathbb{E}}_{7} & \textup{if $(a_1,a_2)=(3,3)$};\\
\widetilde{ \mathbb{E}}_{6} & \textup{if $(a_1,a_2)=(2,4)$};
  \end{array}
 \] 
 and it is of wild acyclic type in all other cases.
\end{theorem}

\begin{proof}
    The quiver $\mu_{\mkern-1.5mu R}\, Q$ is given by \cref{fig:mutation_R} together with the branches $s_1 \to \cdots \to s_{a_1-1} \to r_{a_2}$ and $r_0 \to t_1 \to \cdots \to t_{a_1-1}$, and is shown in \cref{fig:mutation_ST}.
\end{proof}

\begin{remark}
    The quiver $\mu_T \mu_{\mkern-1.5mu R}\, Q$ is the quiver $T_{p,q,r}$ of \cite{Arnold, DW}, with $p = q = a_1+1$ and $r=a_2-1$.
\end{remark}

\begin{lemma}\label{lem 7.4}
    We have
    \[\mu_S(\mu_{\mkern-1.5mu R}\, T) = \mu_{S}^{-1}(\mu_{\mkern-1.5mu R}\, T),\]
    and dually
    \[\mu_T(\mu_{\mkern-1.5mu R}\, T) = \mu_{T}^{-1}(\mu_{\mkern-1.5mu R}\, T).\]
\end{lemma}

\begin{proof}
    In $\mu_{\mkern-1.5mu R} \,Q$, there are no arrows between $s_i$ and $s_{i + \ell}$ for $\ell > 1$. Therefore
    \[\mu_{s_i}\mu_{s_{i+\ell}}=\mu_{s_{i+\ell}} \mu_{s_i}.\] Thus the mutation sequence \[\mu_{s_1}\mu_{s_2}\mu_{s_3} \cdots \mu_{r_{a_2}} \circ \cdots \circ \mu_{s_1}\mu_{s_2}\mu_{s_3} \circ \mu_{s_1}\mu_{s_2} \circ \mu_{s_1}\] 
    is equal to the mutation sequence
    \[\mu_{s_1} \circ \mu_{s_2}\mu_{s_1} \circ \mu_{s_3}\mu_{s_2}\mu_{s_1} \cdots \circ \mu_{r_{a_2}} \cdots \mu_{s_3}\mu_{s_2}\mu_{s_1},\]
    and so we have $\mu_S(\mu_{\mkern-1.5mu R}\, T) = \mu_{S}^{-1}(\mu_{\mkern-1.5mu R}\, T)$.
\end{proof}

\begin{corollary}\label{cor:order_two}
    The mutation sequence $\mu$ is of order two, i.e.,
    \[\mu T = \mu^{-1}T.\]
\end{corollary}

\begin{proof}
    There are no arrows between $s_i$ and $t_j$ for any $1 \leq i,j \leq a_1-1$, so $\mu_S \mu_T = \mu_T \mu_S$. 
    Thus
    $  \mu^2 = \mu_{\mkern-1.5mu R}^{-1}  \mu_S  \mu_T  \mu_{\mkern-1.5mu R}
    \ 
    \mu_{\mkern-1.5mu R}^{-1}  \mu_S  \mu_T  \mu_{\mkern-1.5mu R}
    =  \mu_{\mkern-1.5mu R}^{-1}  \mu_S\mu_S   \mu_T \mu_T \mu_{\mkern-1.5mu R}
    $,
    and by Lemma~\ref{lem 7.4}, this is equal to $ \mu_{\mkern-1.5mu R}^{-1}  \mu_{\mkern-1.5mu R}$.
\end{proof}

\begin{theorem}
    The correspondence $T \mapsto A[1]$ induces a cluster automorphism $\sigma$. The automorphism $\sigma$ has order two, and is given by the sequence of mutations $\mu$.
\end{theorem}

\begin{proof}
    From \cref{cor:order_two}, it is enough to show that $\sigma$ is given by the sequence of mutations $\mu$. We show that $\mu_T \mu_S \mu_{\mkern-1.5mu R}\, T = \mu_{\mkern-1.5mu R}\, A[1]$.
    
    First, we apply the mutation sequence $\mu_{\mkern-1.5mu R}$ to $T$. As shown in \cref{fig:mutation_R}, at each step prior to mutating at the vertex $r_i$ we have exactly one arrow ending at $r_i$, namely $r_0 \to r_i$. Hence we replace the summand $M(r_i)$ of $T$ with the module $X(r_i)$, given by the exchange triangle
    \begin{align*}
        M(r_0)[-1] \to X(r_i)[-1] \to M(r_i) \xrightarrow{f} M(r_0) \to X(r_i),
    \end{align*}
    where $f$ is an $\operatorname{add} (T/M(r_i))$-approximation.
    By \cref{rem:proj_inj}, we know that $M(r_0)$ is injective. Therefore the image of this triangle under the functor $\operatorname{Hom}_{\mathcal{C}}(T,-)$ is the exact sequence in $\operatorname{Mod}A$:
    \begin{align*}
        0 \to \tau^{-1}X(r_i) \to M(r_i) \xrightarrow{f} M(r_0).
    \end{align*}
    So $\tau^{-1}X(r_i)$ is given by the kernel of the morphism $M(r_i) \xrightarrow{f} M(r_0)$, hence
    \begin{align*}
        \tau^{-1}X(r_i) = \begin{smallmatrix} & r_0 & \\ r_1 & & t_1 \\ r_2 & & t_2 \\ \vdots & & \vdots \\ r_{i-1} & & t_{a_1-1}\end{smallmatrix}.
    \end{align*}
    Now $\tau^{-1}X(r_i)$ has a projective resolution $\cdots \to P(r_i) \xrightarrow{g} P(r_0) \to \tau^{-1}X(r_i) \to 0$, so by applying the Nakayama functor $\nu$ we see that $X(r_i)$ is given by the kernel of the map $I(r_i) \xrightarrow{\nu g} I(r_0)$. Hence
    \begin{align*}
        X(r_i) = \begin{smallmatrix} r_1 \\ r_2 \\ \vdots \\ r_i \end{smallmatrix}, \quad\quad \text{and} \quad\quad \mu_{\mkern-1.5mu R}\, T = T \setminus \Bigg(\bigoplus_{i=1}^{a_2-1}M(r_i)\Bigg) \oplus \Bigg(\bigoplus_{i=1}^{a_2-1} X(r_i)\Bigg) .
    \end{align*}
    Second, we apply the mutation sequence $\mu_S$. Note that the mutation sequence $\mu_S$ only mutates the quiver $\mu_{\mkern-1.5mu R}\, Q$ at sources. Since the quiver of $T$ is $Q^{\operatorname{op}}$, each mutation in this sequence acts on the corresponding summand of $T$ as the shift operator. Dually, each mutation in the sequence $\mu_T$ acts on the corresponding summand of $T$ as the inverse shift operator.
    
    Let $s_{a_1} := r_{a_2}$ and $t_0 := r_0$. For each $i$, the mutation sequence $\mu_S$ mutates the vertex $s_i$ exactly $a_1 - i + 1$ times and the mutation sequence $\mu_T$ mutates the vertex $t_i$ exactly $i+1$ times. We conclude that
    
    \begin{align*}
        \mu_T\mu_S\mu_{\mkern-1.5mu R}\, T = \Bigg(\bigoplus_{i=1}^{a_2-1} X(r_i)\Bigg) \oplus \Bigg(\bigoplus_{i=1}^{a_1} M(s_i)[a_1-i+1]\Bigg) \oplus \Bigg( \bigoplus_{i=0}^{a_1-1} M(t_i)[-(i+1)]\Bigg).
    \end{align*}
    
    Next, we compute $M(s_i)[a_1-i+1]$ and $M(t_i)[-(i+1)]$ for all $i$. Given a pair $(j, \ell)$ such that $1 \leq j \leq \ell < a_1-1$, define
        \begin{align*}
            L := \begin{smallmatrix}t_j \\ t_{j+1} \\ \vdots \\ t_{\ell} \end{smallmatrix}.
        \end{align*}
        Then $L$ has a projective resolution $\cdots \to P(t_{\ell+1}) \xrightarrow{g} P(t_j) \to L \to 0$. Applying the Nakayama functor $\nu$, we see that $\tau L = L[1]$ is given by the kernel of the morphism $I(t_{\ell + 1}) \xrightarrow{\nu g} I(t_j)$, so
        \begin{align*}
            L[1] = \begin{smallmatrix} t_{j+1} \\ t_{j+2} \\ \vdots \\ t_{\ell+1} \end{smallmatrix}.
        \end{align*}
        Recall $\tau M(s_i)$ from \cref{lem:tau_M(si)}. Then for all $i$, it follows that
        \begin{align*}
            M(s_i)[a_1-i+1] &= \tau M(s_i)[a_1 - i - 1][1], \\[1.5em]
            &= \begin{smallmatrix}t_1\\t_2 \\\vdots\\t_i\end{smallmatrix}[a_1 - i - 1][1], \\[1.5em]
            &= \begin{smallmatrix}t_{a_1-i}\\t_{a_1-i+1} \\\vdots\\t_{a_1-1}\end{smallmatrix}[1], \\[1.5em]
            &= P(t_{a_1-i})[1].
        \end{align*}
        Then from a dual argument, we also have
        \begin{align*}
            M(t_i)[-(i+1)] = I(s_{a_1-i})[-1] = P(s_{a_1-i})[1].
        \end{align*}
        Hence
        \begin{align*}
            \mu_T \mu_S \mu_{\mkern-1.5mu R}\, T = \Bigg(\bigoplus_{i=1}^{a_2-1}X(r_i)\Bigg)\oplus \Bigg(\bigoplus_{i=0}^{a_1-1}P(s_i)[1]\Bigg) \oplus \Bigg(\bigoplus_{i=1}^{a_1}P(t_i)[1]\Bigg).
        \end{align*}
        
        Finally, let $A[1] = \bigoplus_{x \in Q_0} P(x)[1]$ with quiver $Q^{\operatorname{op}}$. For $i=1,2, \cdots, a_2-1$, let $X'(r_i)$ be the module replacing the summand $P(r_i)[1]$ in the module $\mu_{\mkern-1.5mu R}\, A[1]$. To complete the proof, we show that $X'(r_i) = X(r_i)$. This proves that $\mu_S \mu_T \mu_{\mkern-1.5mu R}\, T = \mu_{\mkern-1.5mu R}\, A[1]$, and thus $\mu T = A[1]$.
        
        Recall from \cref{fig:mutation_R} that at each step of $\mu_{\mkern-1.5mu R}$, prior to mutating at the vertex $r_i$ we have exactly one arrow ending at $r_i$, namely $r_0 \to r_i$. So $X'(r_i)$ is given by the exchange triangle
        \begin{align*}
            P(r_i) \xrightarrow{h} P(r_0) \to X'(r_i)[-1] \to P(r_i)[1] \to P(r_0)[1] \to X'(r_i).
        \end{align*}
        The image of this triangle under the functor $\operatorname{Hom}_{\mathcal{C}}(A[1],-)$ is the exact sequence in $\operatorname{mod}A$:
        \begin{align*}
            P(r_i) \xrightarrow{h} P(r_0) \to \tau^{-1}X'(r_i) \to 0.
        \end{align*}
        We have a projective resolution $\cdots \to P(r_i) \xrightarrow{h} P(r_0) \to \tau^{-1}X'(r_i)$, so applying the Nakayama functor $\nu$ tells us that $X'(r_i)$ is given by the kernel of the map $I(r_i) \xrightarrow{\nu h} I(r_0)$. This is exactly how we computed $X(r_i)$, so $X'(r_i) = X(r_i)$. This completes the proof.
\end{proof}

\section{Example}\label{sec:example}
We illustrate the results in the example $(a_1,a_2)=(2,2)$ which is of type $\mathbb{A}_5$. The corresponding knot is the figure eight knot $L[2,2]$. We have
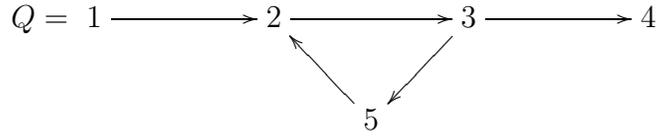
\[Q=\xymatrix{1\ar[rr]&&2\ar[rr]&&3\ar[rr]&&4\\&&&5\ar[lu]\ar@{<-}[ru]}\] and the Jacobian algebra $A$ is the path algebra of  $Q$ modulo the two-sided ideal generated by the subpaths of length 2 in the 3-cycle. Thus $A$ is a cluster-tilted algebra of Dynkin type $\mathbb{A}_5$. 

The $A$-module $T=\oplus_{i=1}^5 M(i)$ is given by  \[
M(1)=\begin{smallmatrix}3\\5 \end{smallmatrix} 
\quad,\quad
M(2) = \begin{smallmatrix}3\\4\ 5 \end{smallmatrix} 
\quad,\quad
M(3) = \begin{smallmatrix}1\ 5\\ 2 \end{smallmatrix} 
\quad,\quad
M(4)=\begin{smallmatrix} 5\\ 2 \end{smallmatrix} 
\quad,\quad
M(5)= \begin{smallmatrix}1\\ 2\\ 3\\ 4 \end{smallmatrix}.
\]
The indecomposable summands $M(2),M(3),M(5)$ correspond to the vertices on the 3-cycle. Each of these indecomposables has precisely 5 submodules, as expected, since 5 is the number of terms in the Jones polynomial of the figure eight knot.
For example, the submodules of $M(2)$ are $0, 4, 5, 4\oplus 5$ and $M(2)$.

 The Auslander-Reiten quiver of the cluster category is the following
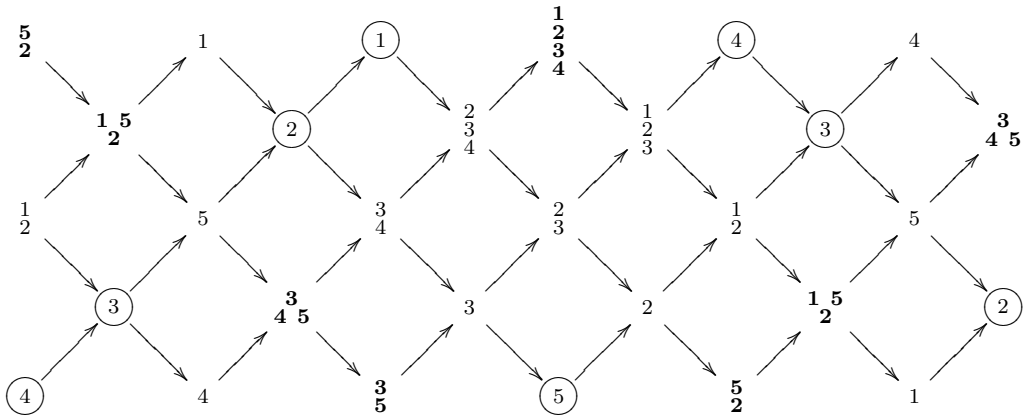
\[\xymatrix@!@R1pt@!@C1pt{
{\begin{smallmatrix}\bf 5 \\ \bf 2\end{smallmatrix}}\ar[dr] 
&&{\begin{smallmatrix}1\end{smallmatrix}\ar[dr]} &&
*+[Fo]{\begin{smallmatrix}{1}\end{smallmatrix}}\ar[dr] &&
{\begin{smallmatrix}\bf 1\\\bf 2\\\bf 3\\\bf 4 \end{smallmatrix}\ar[dr]} &&
*+[Fo]{\begin{smallmatrix}{4}\end{smallmatrix}}\ar[dr] &&{\begin{smallmatrix}{4}\end{smallmatrix}}\ar[dr] 
\\
&{\begin{smallmatrix}\bf 1\ 5\\\bf 2\end{smallmatrix}}\ar[dr]\ar[ur] &&
*+[Fo]{\begin{smallmatrix}{2}\end{smallmatrix}}\ar[ur]\ar[dr] &&
{\begin{smallmatrix}2\\3\\4\end{smallmatrix}}\ar[dr]\ar[ur] &&
{\begin{smallmatrix}1\\2\\3\end{smallmatrix}}\ar[dr]\ar[ur] &&
*+[Fo]{\begin{smallmatrix}{3}\end{smallmatrix}}\ar[dr]\ar[ur] &&
{\begin{smallmatrix}\bf 3\\\bf 4\ 5\end{smallmatrix}} 
\\
{\begin{smallmatrix}1\\2\end{smallmatrix}}\ar[dr]\ar[ur]
&&{\begin{smallmatrix}5\end{smallmatrix}}\ar[dr]\ar[ur] &&
{\begin{smallmatrix}3\\4\end{smallmatrix}}\ar[dr]\ar[ur] &&
{\begin{smallmatrix}2\\3\end{smallmatrix}}\ar[dr]\ar[ur] &&
{\begin{smallmatrix}1\\2\end{smallmatrix}}\ar[dr]\ar[ur] &&
{\begin{smallmatrix}5\end{smallmatrix}}\ar[dr]\ar[ur]
\\
&*+[Fo]{\begin{smallmatrix}{3}\end{smallmatrix}}\ar[ur]\ar[dr] &&
{\begin{smallmatrix}\bf 3\\\bf 4\ 5\end{smallmatrix}}\ar[dr]\ar[ur] &&
{\begin{smallmatrix}3\end{smallmatrix}}\ar[dr]\ar[ur] &&
{\begin{smallmatrix}2\end{smallmatrix}}\ar[dr]\ar[ur] &&
{\begin{smallmatrix}\bf 1\ 5\\\bf 2\end{smallmatrix}}\ar[dr]\ar[ur] &&
*+[Fo]{\begin{smallmatrix}{2}\end{smallmatrix}}
\\
*+[Fo]{\begin{smallmatrix}{4}\end{smallmatrix}}\ar[ur] &&
{\begin{smallmatrix}{4}\end{smallmatrix}}\ar[ur] &&
{\begin{smallmatrix}\bf 3 \\ \bf 5\end{smallmatrix}}\ar[ur] &&
*+[Fo]{\begin{smallmatrix}{5}\end{smallmatrix}}\ar[ur] &&
{\begin{smallmatrix}\bf 5 \\ \bf 2\end{smallmatrix}}\ar[ur] &&
{\begin{smallmatrix}{1}\end{smallmatrix}}\ar[ur] 
}
\] 
where vertices with the same labels are identified, and the shift of the projective $P(i)$ is denoted by an $i$ in a circle. The summands of $T$ are set in bold face.

The mutation sequence $\mu$ of Section~\ref{sec:mutation} is $\mu=\mu_5\circ\mu_4\circ\mu_3\circ\mu_4\circ\mu_1\circ\mu_2\circ\mu_1\circ\mu_5$. This mutation sequence  sequence sends $T$ to $A[1]$, more precisely
\[M(1)\mapsto P(4)[1], \ 
M(2)\mapsto P(3)[1], \ 
M(3)\mapsto P(2)[1],\ 
M(4)\mapsto P(1)[1],\ 
M(5)\mapsto P(5)[1].
\]


\end{document}